\documentclass[12pt]{article}

\usepackage{color}

\usepackage{amssymb,amsthm,amsmath}
\usepackage{latexsym}
\usepackage{graphicx}
\usepackage{url}
\usepackage{fullpage}

\newtheorem{theorem}{Theorem}
\newtheorem{lemma}[theorem]{Lemma}
\newtheorem{claim}[theorem]{Claim}

\newtheorem{obs}[theorem]{Observation}

\newtheorem{prop}[theorem]{Proposition}

\newcommand\ex{\ensuremath{\mathrm{ex}}}

\title{Extremal results for Berge-hypergraphs}

\date{}

\author{
D\'aniel Gerbner\thanks{Hungarian Academy of Sciences, Alfr\'ed R\'enyi Institute of Mathematics, P.O.B. 127, Budapest H-1364, Hungary. Research supported by OTKA grant PD-109537.}
\and
Cory Palmer\thanks{Department of Mathematical Sciences,
University of Montana, Missoula, Montana 59812, USA. Research supported by University of Montana UGP grant 325341.}
}

\begin{document}

\maketitle

\begin{abstract}
Let $G$ be a graph and $\mathcal{H}$ be a hypergraph both on the same vertex set. We say that a hypergraph $\mathcal{H}$ is a \emph{Berge}-$G$ if there is a bijection
$f : E(G) \rightarrow  E(\mathcal{H})$ such that for $e \in E(G)$ we have $e \subset f(e)$.
This generalizes the established definitions of ``Berge path'' and ``Berge cycle'' to general graphs.
For a fixed graph $G$ we examine the maximum possible size (i.e.\ the sum of the cardinality of each edge) of a hypergraph with no Berge-$G$ as a subhypergraph. 
In the present paper we prove general bounds for this maximum when $G$ is an arbitrary graph. We also consider the specific case when $G$ is a complete bipartite graph and
prove an analogue of the K\H ov\'ari-S\'os-Tur\'an theorem.
\end{abstract}
\section{Introduction}

Let $G$ be a graph and $\mathcal{H}$ be a hypergraph both on the same vertex set. We say that the hypergraph $\mathcal{H}$ is a \emph{Berge}-$G$ if there is a bijection
$f : E(G) \rightarrow  E(\mathcal{H})$ such that for $e \in E(G)$ we have $e \subset f(e)$. In other words, given a graph $G$ we can construct a Berge-$G$ by replacing each edge of $G$ with a hyperedge that contains it. Alternatively, a hypergraph $\mathcal{H}$ is a Berge-$G$ if each hyperedge $h$ in $\mathcal{H}$ can be mapped to two vertices contained in $h$ such that the resulting graph is $G$ (where we are allowed to ignore extra isolated vertices).
Note that when $G$ is a cycle, this definition corresponds to the established definitions of ``Berge path'' and ``Berge cycle.'' Note that for a fixed graph $G$, a Berge-$G$ is a class of hypergraphs.

We say that a hypergraph $\mathcal{H}$ \emph{contains} the graph $G$ if $\mathcal{H}$ has a Berge-$G$ as a subhypergraph. If $\mathcal{H}$ contains no $G$, then we say that $\mathcal{H}$ is $G$-\emph{free}.
We would like to examine the maximum number of edges possible in a hypergraph that is $G$-free and has $n$ vertices. Throughout most of this paper we allow our hypergraphs to include multiple copies of the same hyperedge (multi-hyperedges). 
A hypergrpah is \emph{simple} if there are no duplicate hyperedges.
For ease of notation we consider a hypergraph $\mathcal{H}$ as a family of edges. Thus $h \in \mathcal{H}$ means $h$ is a hyperedge of $\mathcal{H}$ and $|\mathcal{H}|$ is the number
of hyperedges in $\mathcal{H}$. 

In \cite{GyKaLe}, Gy\H ori, Katona, and Lemons proved the following analogue of the Erd\H os-Gallai theorem \cite{ErGa} for Berge paths.

\begin{theorem}[Gy\H ori, Katona, Lemons \cite{GyKaLe}]
Let $P_{k+1}$ be a path on $k+1$ vertices.
Suppose $\mathcal{H}$ is a $P_{k+1}$-free $m$-uniform hypergraph on $n$ vertices and $m >2$. If $k>m$, then
$$|\mathcal{H}| \leq \frac{n}{k} \binom{k}{m}.$$
If $k \leq m$, then
$$|\mathcal{H}| \leq \frac{n(k-1)}{m+1}.$$
\end{theorem}

Gy\H{o}ri and Lemons \cite{GyLe4} proved that if an $m$-uniform hypergraph on $n$ vertices is $C_{2k}$-free and $m \ge 3$, then it contains at most $O(n^{1+1/k})$ edges, which matches the bound found in the graph case (see the even cycle theorem of Bondy and Simonovits \cite{BoSi}). They also showed the surprising fact that the maximum number of edges in a $C_{2k+1}$-free $m$-uniform $n$-vertex hypergraph is also $O(n^{1+1/k})$ which is significantly different from the graph case. Moreover, they proved 

\begin{theorem}[Gy\H{o}ri, Lemons \cite{GyLe4}]\label{cycle-free}
Suppose $\mathcal{H}$ is either $C_{2k}$-free or $C_{2k+1}$-free hypergraph on $n$ vertices. If every hyperedge in $\mathcal{H}$ has size at least $4k^2$, then
$$\sum_{v \in V(\mathcal{H})} d(v) = O(n^{1+1/k}).$$
\end{theorem}

The case when all hyperedges are of size $3$ was addressed by Gy\H ori and Lemons \cite{GyLe2} and recently improved by F\"uredi and \"Ozkahya \cite{FuOz}.

Throughout the paper when considering non-uniform hypergraphs $\mathcal{H}$ we are interested in the value of $\sum_{v \in V(\mathcal{H})} d(v) = \sum_{h \in \mathcal{H}} |h|$ instead of $|\mathcal{H}|$. Note that when a hypergraph is uniform than the two parameters can easily be computed from each other.
We examine the size of hypergraphs that are $F$-free when $F$ is a general graph and when $F$ is a complete bipartite graph. 

\begin{theorem}\label{general} 
Let $F$ be a graph and let $\mathcal{H}$ be an $F$-free hypergraph on $n$ vertices. If every hyperedge in $\mathcal{H}$ has size at least $|V(F)|$, then 
$$\sum_{v \in V(\mathcal{H})}d(v) = O(n^2).$$
Furthermore, there exists a $K_r$-free hypergraph $\mathcal{H}$ such that each edge has size $r$ and
$$\sum_{v \in V(\mathcal{H})} d(v) = \Omega(n^2).$$
\end{theorem}

%

The K\H{o}v\'ari-S\'os-Tur\'an theorem \cite{KoSoTu} gives an upper-bound on the Tur\'an number (extremal number) of the complete bipartite graph.

\begin{theorem}[K\H{o}v\'ari, S\'os, Tur\'an \cite{KoSoTu}]\label{KST}
If $G$ is a $K_{s,t}$-free graph on $n$ vertices and $s \leq t$, then there is a constant $C$ depending on $s$ and $t$ such that
$$|E(G)| \leq C n^{2-1/s}.$$
\end{theorem}

We prove the following analogue for $K_{s,t}$-free hypergraphs.

\begin{theorem}\label{bipar}
If $\mathcal{H}$ is a $K_{s,t}$-free hypergraph, $s \le t$ and every hyperedge in $\mathcal{H}$ has size at least $s+t$, then 
$$\sum_{v \in V(\mathcal{H})}d(v) = O(n^{2-1/s}).$$
Furthermore, there exists a $K_{s,t}$-free hypergraph $\mathcal{H}$ such that each edge has size $s+t$ and
$$\sum_{v \in V(\mathcal{H})}d(v) = \Omega( \ex(n,K_{s,t})).$$
\end{theorem}

The case when $s=t=2$ in Theorem~\ref{bipar} is of particular interest as $K_{2,2}$ is also a $C_4$. In \cite{GyLe3} Gy\H ori and Lemons examined the maximum possible value of the sum $\sum_{h \in \mathcal{H}}(|h|-3)$ for $\mathcal{H}$ a $C_4$-free hypergraph.
Observe that this sum is closely connected to the sum $\sum_{v \in V(\mathcal{H})}d(v)$ as $\sum_{h \in \mathcal{H}} (|h|-3) = \sum_{v \in V(\mathcal{H}}d(v) - 3|\mathcal{H}|$. Furthermore, note that the term $|h|-3$ in the sum is necessary as a hypergraph may have arbitrary many copies of an edge of size $3$ and still avoid a $C_4$. Our third main theorem involves this parameter,

\begin{theorem}\label{C4-free}
If $\mathcal{H}$ is a $C_4$-free hypergraph, then 
$$\sum_{h \in \mathcal{H}} (|h|-3) \leq \frac{\sqrt{6}}{2}n^{3/2} + O(n).$$ 
Furthermore, there exists a $C_4$-free hypergraph $\mathcal{H}$ such that 
$$\frac{1}{3\sqrt{3}}n^{3/2} + o(n^{3/2}) \leq \sum_{h \in \mathcal{H}} (|h|-3).$$
\end{theorem}

This improves the results of Gy\H{o}ri and Lemons \cite{GyLe3} who showed that the leading term of the sum above is between $\frac{1}{8}n^{3/2}$ and $12\sqrt{2}n^{3/2}$.

If a hypergraph $\mathcal{H}$ is $C_i$-free for $i=2,3,\dots, g-1$, then we say that $\mathcal{H}$ has \emph{girth} $g$ (in the Berge sense). Note that the property of being $C_2$-free implies that any two hyperedges intersect in at most one vertex.
Lazebnik and Verstra\"ete \cite{LaVe} examined hypergraphs of girth $5$ where all edges are of size $r$. Of particular interest is their result when $r=3$,

\begin{theorem}[Lazebnik, Verstra\"ete \cite{LaVe}]\label{LV-girth}
If $\mathcal{H}$ is a hypergraph on $n$ vertices with girth $5$ and every hyperedge in $\mathcal{H}$ has size $3$ and the maximum number of edges, then
$$|\mathcal{H}| = \frac{1}{6}n^{3/2} + o(n^{3/2}).$$
\end{theorem}

The construction used in the lower bound of Theorem~\ref{C4-free} will be based on the lower bound from Theorem~\ref{LV-girth}.
Furthermore, in the last section we find a new (short) proof of the upper bound in Theorem~\ref{LV-girth}.



\section{$F$-free hypergraphs}\label{2}

\begin{lemma}\label{anygraph} 
Let $F$ be a graph and suppose $\mathcal{H}$ is an $F$-free hypergraph. If every edge of $\mathcal{H}$ has size at least $|V(F)|$,
then $|\mathcal{H}| \le \ex(n,F)$.
\end{lemma}

\begin{proof}
Let us construct a graph $H$ on the ground set of $\mathcal{H}$ by embedding a (unique) edge into each hyperedge of $\mathcal{H}$. By definition if $\mathcal{H}$ is $F$-free, then $H$ must be $F$-free and thus
$|\mathcal{H}|=|H|\le \ex(n,F)$.

We can construct $H$ greedily. Take the hyperedges of $\mathcal{H}$ in arbitrary order and for each hyperedge embed an edge that has not already been used in $H$.
If at some step we cannot find such an edge, then $H$ contains a complete graph of size equal to the size of the hyperedge. Obviously, such a complete graph contains a copy of $F$; a contradiction.
\end{proof}

The following useful anti-Ramsey lemma appears in Babai \cite{Ba} (a generalization appears in Erd\H os, Ne\v set\v ril, and R\" odl \cite{ErNeRo}). We include the very short proof for the sake of completeness.

\begin{lemma}\label{rainbow} 
A properly edge-colored complete graph on at least $r^3$ vertices contains a rainbow $K_r$.
\end{lemma}

\begin{proof} 
Suppose we have a proper edge-coloring of the complete graph on $r^3$ vertices and let $K$ be the largest rainbow sub-complete graph. Then every vertex $x$ not in $K$ is adjacent to some $y \in K$ by an edge with the same color as one of the $\binom{|K|}{2}$ colors appearing in $K$. Thus we have
$$r^3-|K| \leq |K| \binom{|K|}{2} < \frac{1}{2}|K|^3-|K|.$$
Therefore
$$|K|> r.$$
\end{proof}

\begin{lemma}\label{edge-sum}
Let $F$ be a graph on $r$ vertices. If $\mathcal{H}$ is an $F$-free hypergraph on $n$ vertices such that each hyperedge is of size at least $r^3$, then
$$\sum_{h \in \mathcal{H}} |h| \le 2\binom{n}{2} + r^3|\mathcal{H}|.$$
\end{lemma}


\begin{proof}
Similar to the proof of Lemma~\ref{anygraph} we will embed a matching of size at least
$$\frac{1}{2}\left(|h|-r^3\right)$$ 
into each hyperedge $h$ in $\mathcal{H}$ such that the collection of matchings are pairwise disjoint (i.e.\ no edge is in more than one matching).

Take the hyperedges of $\mathcal{H}$ in arbitrary order and suppose
we have reached a hyperedge $h$ where we cannot find a large enough matching to embed, i.e., we can only embed a matching of size less than $\frac{1}{2}(|h|-r^3)$ .
Then there is a set $K$ of more than $r^3$ vertices such that 
every edge in $K$ has already been used by a matching from previous hyperedges.
If we color each of these matchings by a unique color corresponding to their hyperedge, then we have a proper edge-coloring of the complete graph on $K$.
As $K$ has at least $\frac{1}{2}r^3$ vertices, by Lemma~\ref{rainbow} we have that $K$ contains a rainbow $K_r$ which obviously contains a rainbow $F$.
However, this rainbow $F$ would correspond to a Berge-$F$ in $\mathcal{H}$; a contradiction.

Therefore, we can embed a matching of size $\frac{1}{2}(|h|-r^3)$ into each hyperedge of $\mathcal{H}$. The hypergraph $\mathcal{H}$ is on $n$ vertices, so there
are at most $\binom{n}{2}$ total edges embedded into the hyperedges of $\mathcal{H}$. Thus
$\sum_{h \in \mathcal{H}} \frac{1}{2}\left(|h|-r^3\right) \leq \binom{n}{2}$.
Rearranging terms gives the lemma.
\end{proof}



\begin{proof}[Proof of Theorem \ref{general}]
First we split $\mathcal{H}$ into two parts. Let $\mathcal{H}_1$ 
be the collection of hyperedges each of size less than
 $r^3$ and let $\mathcal{H}_2 = \mathcal{H} - \mathcal{H}_1$ be the remaining edges.

All hyperedges in $\mathcal{H}$ are of size at least $r$, so by Lemma~\ref{anygraph} we have $|\mathcal{H}_1| \leq \ex(n,F)$ and
$|\mathcal{H}_2| \leq \ex(n,F) \leq n^2$.

Therefore,
\begin{align*}
\sum_{h \in \mathcal{H}_1} |h| & \le r^3|\mathcal{H}_1| = O(n^2)
\end{align*}
On the other hand, by Lemma~\ref{edge-sum}
\begin{align*}
\sum_{h \in \mathcal{H}_2} & |h| \leq 2\binom{n}{2} + r^3|\mathcal{H}_2| = O(n^2). 
\end{align*}

Combining the two estimates above gives,
$$\sum_{v \in V(\mathcal{H})}d(v)=\sum_{h \in \mathcal{H}} |h| = \sum_{h \in \mathcal{H}_1} |h| + \sum_{h \in \mathcal{H}_2} |h| = O(n^2).$$ 

To complete the proof of Theorem~\ref{general} we need to construct a $K_{r}$-free hypergraph with $\Omega(n^2)$ edges of size $r$.

Suppose $G$ is an $n$-vertex $K_{r}$-free graph with $\Omega(n^2)$ edges (e.g.\ a Tur\'an graph). Let $G'$ be a bipartite subgraph of $G$ with at least half of the total edges and let $A'$ and $B'$ be the partition classes of $G'$. 
Let us replace each vertex $a_i \in A'$ with a set $A_i$ of $r-1$ copies of $a_i$ to get an $r$-uniform hypergraph $\mathcal{H}$. Put $A = \cup_i A_i$ and $B = B'$. 
Observe that no hyperedge of $\mathcal{H}$ contains two vertices from $B$ nor does it contain vertices from two distinct $A_i$s. Therefore, the vertices of a $K_r$ in $\mathcal{H}$ can only be those of some $A_i$ and a single vertex $b$ from $B$. However, for each vertex in $B$ there is only one hyperedge of $\mathcal{H}$ containing it and $A_i$, so these vertices cannot be part of a Berge-$K_r$ in $\mathcal{H}$.

Therefore $\mathcal{H}$ is a $K_r$-free $r$-uniform hypergraph on at most $(r-1)n$ vertices and $\Omega(n^2)$ edges, i.e., 
$$\sum_{v \in V(\mathcal{H})} d(v) = r|\mathcal{H}| = \Omega(n^2) = \Omega(|V(\mathcal{H})|^2).$$
\end{proof}

Before proving Theorem~\ref{bipar} we need an anti-Ramsey lemma similar to Lemma~\ref{rainbow} but for complete bipartite graphs. This lemma was also
proved by Cho, Defant, Sonneborn \cite{ChDeSo}.

\begin{lemma}\label{rainbowbipar} 
A properly edge-colored $K_{s,s(s-1)(t-1)+t}$ contains a rainbow $K_{s,t}$.
\end{lemma}

\begin{proof}
Suppose $K_{s,s(s-1)(t-1)+t}$ is properly edge-colored and let $S$ be the class of size $s$ and $T$ be the other class.
Consider the largest rainbow $K_{s,t'}$ that uses all $s$ vertices of $S$. That is, there are $t'$ vertices in $T$ such that all edges between them and $S$ are of distinct colors. It is enough to show that $t' \geq t$.

Obviously $t' \ge 1$ as the edge-coloring is proper. For every vertex $x$ not in $K_{s,t'}$ there is a neighbor of $x$ in $S$ such that the color of the edge between them is one of the colors used in $K_{s,t'}$ (otherwise we can add $x$ to $K_{s,t'}$).
Each vertex $y\ in S$ is adjacent to $t'$ vertices in $K_{s,t'}$, so there are $st'-t' = (s-1)t'$ colors
that appear in $K_{s,t'}$ not on edges incident to $y$. So at most $y$ can force $(s-1)t'$ many vertices in $T$ to not be in $K_{s,t'}$.
Therefore, in total, there are at most $s(s-1)t'$ vertices in $T$ that are not in $K_{s,t'}$.
So $T$ has at most $s(s-1)t' + t'$ vertices, hence $t'\ge t$.
\end{proof}

\begin{proof}[Proof of Theorem \ref{bipar}] 

First we partition the hyperedges of $\mathcal{H}$ into two parts. Let $\mathcal{H}_1$ be the collection
of hyperedges each of size at least $s+s(s-1)(t-1)+t$ and let $\mathcal{H}_2 = \mathcal{H} - \mathcal{H}_1$ be the remaining hyperedges.

\begin{claim}
$$\sum_{h \in \mathcal{H}_1} |h|= O(n^{2-1/s}).$$
\end{claim}

%

\begin{proof}\renewcommand{\qedsymbol}{$\blacksquare$}
We would like to construct a simple graph $G$ by embedding into each hyperedge $h$ of $\mathcal{H}_1$ a matching of size 
$$\frac{1}{2}\left(|h| - s - s(s-1)(t-1)-t\right)$$
 such that the collection of these matchings are pairwise disjoint (i.e.\ no edge is used more than once).
As before, we take the edges of $\mathcal{H}_1$ in arbitrary order.
Suppose that we have reached a hyperedge $h$ 
to which we cannot embed a matching of the desired size.
This implies that $h$ contains
a complete graph $K$ on $s + s(s-1)(t-1)+t$ many vertices. 
If we color the edges of this complete graph according to their
corresponding hyperedge, then we have a proper edge-coloring of $K$.
In particular, $K$ contains a properly edge-colored $K_{s,s(s-1)(t-1)+t}$.
Thus, by Lemma~\ref{rainbowbipar}, $K$ contains a rainbow $K_{s,t}$
which corresponds to a Berge-$K_{s,t}$ in $\mathcal{H}_1$; a contradiction.

For the same reason, $G$ must be $K_{s,s(s-1)(t-1)+t}$-free, thus, applying the
K\H ov\'ari-S\' os-Tur\' an theorem (Theorem~\ref{KST}) we have
$$|E(G)| = O(n^{2-1/s}).$$
Furthermore, as $\mathcal{H}_1$ is $K_{s,t}$-free, Lemma~\ref{anygraph} gives
$$|\mathcal{H}_1| \leq \ex(n,K_{s,t}) = O(n^{2-1/s}).$$

By construction, the number of edges in $G$ is
$$|E(G)| = \sum_{h \in \mathcal{H}_1} \frac{1}{2} \left(|h|-s-s(s-1)(t-1)-t\right) 
= \frac{1}{2}\left(\sum_{h \in \mathcal{H}_1} |h|\right) - \frac{1}{2}\left(s+s(s-1)(t-1)+t\right)|\mathcal{H}_1|.$$

Rearranging terms we get
$$\sum_{h \in \mathcal{H}_1} |h| = 2|E(G)| + (s+(s-1)(t-1)+t)|\mathcal{H}_1| = O(n^{2-1/s}).$$
\end{proof}

Again, as $\mathcal{H}_2$ is $K_{s,t}$-free, Lemma~\ref{anygraph} gives
$$|\mathcal{H}_2| \leq \ex(n,K_{s,t}) = O(n^{2-1/s}).$$
Every hyperedge in $\mathcal{H}_2$ has size less than $s+s(s-1)(t-1)+t$, so clearly
$$\sum_{h \in \mathcal{H}_2} |h| \leq (s+s(s-1)(t-1)+t)|\mathcal{H}_2| = O(n^{2-1/s}).$$
Therefore,
$$\sum_{h \in \mathcal{H}} |h| = 
\sum_{h \in \mathcal{H}_1} |h| + \sum_{h \in \mathcal{H}_2} |h| = O(n^{2-1/s}).$$




%


To complete the proof of Theorem~\ref{bipar} we need to construct a $K_{s,t}$-free hypergraph with edges of size $s+t$ of the appropriate size.
Let $s \leq t$. If $s = 1$, then $K_{s,t}$ is a star. In this case, $\ex(n,K_{1,t})$ is linear in $n$ and it is easy to construct $1+t$-uniform hypergraphs with no Berge-$K_{1,t}$
that satisfy the theorem.
Now we suppose $s \geq 2$ and let $G$ be a $K_{s,t}$-free graph. Let $G'$ be a bipartite subgraph of $G$ with at least half of the total edges and let $A'$ and $B'$ be the partition classes of $G'$. 
Let us replace each vertex $a_i \in A'$ with a set $A_i$ of $s+t-1$ copies of $a_i$ to get a $(s+t)$-uniform hypergraph $\mathcal{H}$. Put $A = \cup_i A_i$ and $B = B'$. Now it remains to confirm that $\mathcal{H}$ has no Berge-$K_{s,t}$.

Suppose $\mathcal{H}$ contains a Berge-$K_{s,t}$, then we can map the hyperedges of the Berge-$K_{s,t}$ into a copy of the graph $K_{s,t}$
such that each edge is contained in the hyperedge that was mapped to it. Let $S$ and $T$ be the classes of $K_{s,t}$.
If $S \subset A$ and $T \subset B$ (or vice versa), then each vertex of $S$ must be in a distinct $A_i$ as each $A_i$ and each vertex of $B$ are contained in exactly one edge.
In this case, $G'$ contains a $K_{s,t}$; a contradiction.

Now either $A$ or $B$ contains a vertex $x \in S$ and a vertex $y \in T$. The vertices $x$ and $y$ are contained in a hyperedge and as $\mathcal{H}$ has no hyperedges containing two vertices in $B$
we must have that $A$ contains $x$ and $y$. No hyperedge of $\mathcal{H}$ contains vertices from $A_i$ and $A_j$ for $i \neq j$, so $x$ and $y$ must both be contained in some $A_i$.
Every vertex of $S$ and $T$ must be contained in a hyperedge with $y$ or $x$, thus each vertex of $S$ and $T$ must be in $A_i$ or $B$. As the size of $A_i$ is $s+t-1$, there must be at least one vertex $z \in S \cup T$ in $B$. 
There is exactly one hyperedge of $\mathcal{H}$ that contains $z$ and any other vertex of $A_i \cup B$. However, the degree of $z$ in the Berge-$K_{s,t}$ is at least $s \geq 2$; a contradiction.
\end{proof}

\section{$C_4$-free hypergraphs}

In this section we prove Theorem~\ref{C4-free}. We begin with a short proof of a weaker upper bound that uses the ideas in Section~\ref{2}.

\begin{prop}
If $\mathcal{H}$ is a $C_4$-free hypergraph, then 
$$\sum_{h \in \mathcal{H}} (|h|-3) \leq \sqrt{3}n^{3/2} + O(n).$$
\end{prop}

\begin{proof}
We will construct a graph $H$ on the vertex set of $\mathcal{H}$ by replacing each hyperedge $h$ with $\lceil \frac{|h|-3}{2} \rceil$  independent edges on the vertices of $h$. We proceed through the hyperedges in any order and assume that we have colored each matching with a unique color. Note that we may assume all edges of $\mathcal{H}$ are of size at least $4$.

Suppose we are replacing the hyperedge $h$ with a matching of color red. Consider an arbitrary subset of $4$ vertices of $h$ such that no vertex is incident to a red edge. Some of these vertices may be contained in other hyperedges and thus some of the six pairs may already be included in some other matching. However, it is easy to see that if all six pairs are used in macthings, then $\mathcal{H}$ must contain a $C_4$. Thus there is some unused pair that can be colored red. Repeat this step until every subset of $4$ vertices of $h$ is incident to one red edge. At this point we have a matching of $\lceil \frac{|h|-3}{2} \rceil$ red edges in $h$.

Now suppose that $H$ contains a $K_{2,4}$, then it is colored properly by the embedding above. Therefore, by Lemma~\ref{rainbowbipar}, the $K_{2,4}$ must contain a rainbow $C_4$. This corresponds to a Berge-$C_4$ in $\mathcal{H}$ which is a contradition.
Therefore $H$ must be $K_{2,4}$-free. By the K\H{o}v\'ari-S\'os-Tur\'an theorem (Theorem~\ref{KST}), we have that $H$ has at most $\frac{\sqrt{3}}{2}n^{3/2} + O(n)$ edges. Thus,
$$\sum_{h \in \mathcal{H}} \left\lceil \frac{(|h|-3)}{2} \right \rceil \leq \frac{\sqrt{3}}{2}n^{3/2} + O(n)$$
which gives the proposition.
\end{proof}

For the proof of Theorem~\ref{C4-free} we will need some simple Ramsey-type results. Recall that every $2$-edge-coloring of $K_6$ contains a
monochromatic triangle.
Let $K_4^-$ denote a $4$-clique with an edge removed. 

\begin{lemma}\label{ramsey-lemma} 

\begin{enumerate}
\item A $2$-edge-coloring of $K_5$ contains either a red $K_4^-$ or a blue matching of size $2$.
\item A $2$-edge-coloring of $K_6$ contains either a red $K_4^-$ or a blue triangle or a blue matching of size $3$.
\item A $2$-edge-coloring of $K_7$ contains either a red $K_4^-$ or a blue triangle.
\end{enumerate}
\end{lemma}

\begin{proof}
1. Suppose there is no red $K_4^-$, then among four vertices there are at least two blue edges. If they are independent, we are done. Therefore, suppose the vertex $x$ is incident to two blue edges.
Then, among the other four vertices there are at least two blue edges. Among these four blue edges there must be two that are independent.

\smallskip

2. If there is no blue triangle on $6$ vertices, then there is a red triangle $xyz$.
For every vertex $a,b,c$ not in the triangle, there is at most one red edge going to the triangle as otherwise we have a red $K_4^-$.
First suppose that $a,b,$ and $c$ are each connected by red edges to different vertices of $xyz$. In this case it is easy to find a blue matching of size $3$ (between $xyz$ and $a,b,c$).
Therefore, we may suppose (without loss of generality) that $a$ and $b$ are connected by red edges to the same vertex in $xyz$, say $x$. So, either $y$ or $z$ is connected to $a,b,c$ by only blue edges.
Considering that there is no red $K_4^-$ on $a,b,c,x$, there is at least one blue edge among $a,b,c$. Among these four blue edges there must be a blue triangle.

\smallskip

3. If there is no blue triangle on $7$ vertices, then there is a red triangle $xyz$.
For every vertex $a,b,c,d$ not in the triangle, there is at most one red edge going to the triangle as otherwise we have a red $K_4^-$.
Furthermore, there must be a blue edge among $a,b,c,d$, say $ab$. There is a vertex among $xyz$ that is not connected by a red edge to $a$ or $b$ therefore we have a blue triangle.
\end{proof}

We are now ready to give a proof of Theorem~\ref{C4-free}.

\begin{proof}[Proof of Theorem~\ref{C4-free}]

We begin with the upper bound.
We will construct a graph $H$ on the vertex set of $\mathcal{H}$ similar to the previous section. 
For each hyperedge $h$ we will embed
$|h|-3$ edges on the vertices of $h$ such that the collection of edges in $h$ consists of pairwise disjoint $K_3$s and (up to three) $K_2$s.
Let us proceed through the hyperedges in arbitrary order.
At hyperedge $h$ we embed edges disjointly from the previously embedded edges.
To show that such an embedding is possible we need a simple claim.

\begin{claim}\label{no-5-edges}
Before embedding edges into $h$, on any $4$ vertices in $h$ there are at most $4$ edges already embedded.
\end{claim}

\begin{proof}\renewcommand{\qedsymbol}{$\blacksquare$}
Let $T = \{x,y,z,w\}$ be a set of $4$ vertices in $h$ with at least $5$ edges.
For convenience we say that the edges already embedded into $T$ are each colored corresponding to their hyperedge.
Note that if two incident edges are embedded into $T$ and they have the same color, then the third edge that forms a triangle must also be embedded into $T$ and have the same color.
$T$ must contain a triangle $xyz$ and the edges $wx$ and $wz$. 

First suppose the triangle is monochromatic with color $1$. Then the other two edges must have two new colors $2$ and $3$. Then it is easy to see
that there is a path on three edges with colors $1,2,3$. Therefore, the hyperedges corresponding to colors $1,2,$ and $3$ and $h$ form a Berge-$C_4$.

Now, if there is no monochromatic triangle in $T$, then the edge-coloring is proper. Therefore, the triangle is rainbow. 
Without loss of generality suppose that $xy$ is color $1$, $yz$ is $2$ and $xz$ is $3$. 
If there is a fourth color on either of the other two edges, then there is a path on three edges with three different colors. It is easy to see that the
three hyperedges corresponding to colors $1,2,$ and $3$, together with $h$ form a Berge-$C_4$. 
If only colors $1,2,$ and $3$ are used then, $wx$ must be color $2$ and $wz$ must be color $1$. 
Therefore, the hyperedges corresponding to colors $1$ and $2$ contain all vertices of $T$
and the hyperedge corresponding to color $3$ contains $x$ and $z$.
These three hyperedges and $h$ form a Berge-$C_4$.
\end{proof}

Now we show that the desired embedding is indeed possible. Suppose we are trying to embed edges into the hyperedge $h$. 
At this point some edges may already be embedded into $h$. Recall that we would like to embed $|h|-3$ previously unused edges into $h$ such that
these edges form triangles and (up to three) independent edges.
By Claim~\ref{no-5-edges} there is no $K_4^-$ among the already embedded edges.
So if $h$ has size $4$, there is an available edge to embed.
If $h$ has size $5$, then by Lemma~\ref{ramsey-lemma} there must be an available matching of size $2$. If $h$ has size $6$,
then by Lemma~\ref{ramsey-lemma} there must be an available triangle or matching of size $3$.
If $h$ has size at least $7$, then by Lemma~\ref{ramsey-lemma} there must be an available triangle. So embed the triangle and on the remaining $|h|-3 \geq 4$ vertices
we can repeat the above procedure.

In this way we have embedded at least $|h|-3$ unique edges into each hyperedge of $\mathcal{H}$. The graph spanned by these edges is $H$. 
Thus 
$$\sum_{h \in \mathcal{H}} (|h|-3) = |E(H)|.$$ 
Furthermore, for each hyperedge, these unique edges form a set of independent triangles with at most three independent edges.

\begin{claim}\label{no-K27}
$H$ contains no $K_{2,7}$.
\end{claim}

\begin{proof}\renewcommand{\qedsymbol}{$\blacksquare$}
Suppose $H$ contains a $K_{2,7}$. We will show that there is a Berge-$C_4$ in the original hypergraph $\mathcal{H}$.
As before, let us color the edges of $H$ according to the hyperedge of $\mathcal{H}$ into which they were embedded.
Let $A = \{x,y\}$ be the class of size $2$ and $B=\{z_1,z_2,\dots,z_7\}$ be the class of size $7$.

\smallskip

{\bf Case 1.} For every vertex in $B$ the two incident edges in $K_{2,7}$ are different colors. Observe that there is a color, say $1$, such that only one edge incident to $x$ is color $1$. Call this edge $xz_1$. The edge $yz_1$ is of a different color, say $2$.
Observe that there are at most $5$ other vertices in $B$ connected to $x$ or $y$ by edges of color $1$ or $2$. Let $z_7$ be the remaining vertex in $B$. Clearly, $xz_7$ and $yz_7$ must be two new colors. Thus the cycle $xz_1yz_7$ is rainbow 
which corresponds to a Berge-$C_4$; a contradiction.

\smallskip

{\bf Case 2.} There is a vertex $z_1 \in B$ such that $xz_1$ and $yz_1$ are both the same color, say $1$. In this case, the edge $xy$ is present in the graph $H$ and has color $1$, therefore
there is only one such $z_i$. Thus $xz_2$ and $yz_2$ are colored $2$ and $3$, respectively. So there are at most $3$ other vertices in $B$ connected by an edge of color $2$ to $x$ or $y$ and at most $1$ other vertex in $B$
connected by an edge of color $3$ to $y$. Let $z_7$ be the remaining vertex and observe that $yz_7$ must be of color $4$. The edge $xz_7$ cannot be color $1$, $2$, or $4$. If $xz_7$ is not color $3$, then the cycle $xz_7yz_2$ is rainbow
which corresponds to a Berge-$C_4$; a contradiction. Therefore, $xz_7$ is color $3$. Thus $z_2$ and $z_7$ are both in the hyperedge corresponding to color $3$. This hyperedge together with the three distinct hyperedges containing the edges $z_2x$, $xy$, and $yz_7$
form a Berge-$C_4$.
\end{proof}

By Claim~\ref{no-K27} and the K\H{o}v\'ari-S\'os-Tur\'an theorem (Theorem~\ref{KST}), we have
$$\sum_{h \in \mathcal{H}} (|h|-3) = |E(H)| \leq \frac{\sqrt{6}}{2}n^{3/2} + O(n).$$
This completes the proof of the upper bound.

Now it remains to prove the lower bound.
Lazebnik and Verstra\"ete \cite{LaVe} constructed an $n$-vertex hypergraph $\mathcal{G}$ with all hyperedges of size $3$ and no Berge-$C_2$, Berge-$C_3$, or Berge-$C_4$ (i.e.\ girth $5$ in the Berge sense) and
$$\frac{1}{6}n^{3/2} + o(n^{3/2})$$
many hyperedges (see Theorem~\ref{LV-girth}).

Starting with the above construction, we replace each vertex with three copies of it to get a $9$-uniform hypergraph $\mathcal{H}$.
Now we show that $\mathcal{H}$ has no Berge-$C_4$. If it contains a Berge-$C_4$ then there are four vertices $a,b,c,d$ and four hyperedges such that
$\{a,b\}, \{b,c\},\{c,d\},\{d,a\}$ are in distinct hyperedges. If $a,b,c,d$ are copies of four distinct vertices in $\mathcal{G}$, then $\mathcal{G}$ contains a Berge-$C_4$; a contradiction. Thus,
at least two of $a,b,c,d$ are copies of the same vertex in $\mathcal{G}$. Without loss of generality there are two cases: either $a=b$ or $a=c$ in $\mathcal{G}$. In the first case
it is easy to see that there is either a Berge-$C_3$ or Berge-$C_2$ in $\mathcal{G}$; a contradiction. Similarily, in the second case we will get a Berge-$C_2$; a contradiction.
Therefore $\mathcal{H}$ has no Berge-$C_4$.

If $\mathcal{H}$ has $n$ vertices, then the number of edges in $\mathcal{H}$ is
$$\frac{1}{6}\left( \frac{n}{3} \right)^{3/2} + o(n^{3/2}).$$
Therefore,
$$\sum_{h \in \mathcal{H}} (|h|-3) \geq \frac{1}{3\sqrt{3}}n^{3/2} + o(n^{3/2}).$$
\end{proof}

\section{Linear hypergraphs} 

In the previous sections we restricted our attention to $F$-free hypergraphs with edges of size at least $|V(F)|$. This condition allowed our theorems to hold even for hypergraphs that have multiple copies of the same hyperedge (multi-hyperedges).
A hypergraph with multi-hyperedges can have arbitrarily many copies of a hyperedge of size less than $|V(F)|$ and remain $F$-free.
Therefore, we need to restrict our attention to \emph{simple} hypergraphs, i.e., hypergraphs with no multi-hyperedges.


In the proof of Lemma~\ref{anygraph} the assumption on the size of the hyperdges is crucial. Indeed there is no such general statement 
comparing the number of edges of an $F$-free hypergraph and $\ex(n,F)$ in this case. For example the complete $3$-uniform $3$-partite hypergraph is clearly $K_4$-free but contains $\Theta(n^3)$ many hyperedges. 

%
%

Unfortunately, in general for arbitrary $F$-free simple hypergraphs we know very little. 
One additional condition which can help in this case is if the hypergraph is $C_2$-free, i.e., two edges intersect in at most one point ($C_2$-free hypergraphs are typically called \emph{linear}). 
Given a $C_2$-free hypergraph $\mathcal{H}$ we can form a graph $G$ by replacing each hyperedge $h$ with an unique edge contained in $h$. Clearly the resulting graph $G$ will not contain $F$ as a subgraph. This gives the following simple observation,

\begin{obs} 
Let $F$ be a graph and let $\mathcal{H}$ be an $F$-free and $C_2$-free hypergraph. If every hyperedge in $\mathcal{H}$ has size at least $2$, then $|\mathcal{H}| \le \ex(n,F)$.
\end{obs}

In order to improve the bound, we can attempt to replace each hyperedge of $\mathcal{H}$ with several graph edges. However, this must be done in such a way that $G$ does not contain $F$ as a subgraph.
If $\mathcal{H}$ is $3$-uniform with girth at least $5$, we can create a graph $G$ by replacing each hyperedge $h$ with the three graph edges in $h$. The resulting graph does not contain $C_4$ (otherwise $\mathcal{H}$ is not $C_4$-free).
In this case, using the fact that $\ex(n,C_4) = \frac{1}{2}n^{3/2} + o(n^{3/2})$ we get the upper bound from Theorem~\ref{LV-girth}. We prove the following slightly more general result,

\begin{prop} 
If $\mathcal{H}$ is a $3$-uniform hypergraph with girth $g\geq 5$, then 
$$|\mathcal{H}| \le \frac{1}{3}\ex(n,C_4,C_5,\dots,C_{g-1}).$$
\end{prop}

\begin{proof} 
Let us form a graph $G$ by replacing each hyperedge $h$ with the three graph edges contained in $h$. It remains to show that $G$ does not contain a cycle $C$ of length between $4$ and $g-1$. 
The assumption that $\mathcal{H}$ has girth $g$ implies that $G$ has no cycle of length less than $g$ (not even a triangle) such that all of its edges come from different hyperedges of $\mathcal{H}$. 

Therefore, if $C$ is a cycle of length between $4$ and $g-1$, then $C$ contains two edges from the same hyperedge of $\mathcal{H}$. These two edges must be incident, so call them $ab$ and $bc$. Now we can replace these edges with edge $ac$ to get a shorter cycle. If we repeat this for any pair of edges of $C$ in the same hyperedge of $\mathcal{H}$ we end up with a shorter non-triangle cycle which has each edge from a different hyperedge of $\mathcal{H}$; a contradiction.
\end{proof}

\bibliographystyle{plain}
\bibliography{kst}

\end{document}